\documentclass{article}%
\usepackage{graphicx}
\usepackage{amsmath}
\usepackage{amsfonts}
\usepackage{amssymb}
\usepackage{setspace}
\usepackage[hmargin=3.5cm,vmargin=3.5cm]{geometry}%
\providecommand{\U}[1]{\protect\rule{.1in}{.1in}}
\providecommand{\U}[1]{\protect\rule{.1in}{.1in}}
\providecommand{\U}[1]{\protect\rule{.1in}{.1in}}
\providecommand{\U}[1]{\protect\rule{.1in}{.1in}}
\providecommand{\U}[1]{\protect\rule{.1in}{.1in}}
\providecommand{\U}[1]{\protect\rule{.1in}{.1in}}
\providecommand{\U}[1]{\protect\rule{.1in}{.1in}}
\providecommand{\U}[1]{\protect\rule{.1in}{.1in}}
\providecommand{\U}[1]{\protect\rule{.1in}{.1in}}
\providecommand{\U}[1]{\protect\rule{.1in}{.1in}}
\providecommand{\U}[1]{\protect\rule{.1in}{.1in}}
\providecommand{\U}[1]{\protect\rule{.1in}{.1in}}
\providecommand{\U}[1]{\protect\rule{.1in}{.1in}}
\providecommand{\U}[1]{\protect\rule{.1in}{.1in}}
\providecommand{\U}[1]{\protect\rule{.1in}{.1in}}
\providecommand{\U}[1]{\protect\rule{.1in}{.1in}}
\providecommand{\U}[1]{\protect\rule{.1in}{.1in}}
\providecommand{\U}[1]{\protect\rule{.1in}{.1in}}
\providecommand{\U}[1]{\protect\rule{.1in}{.1in}}
\providecommand{\U}[1]{\protect\rule{.1in}{.1in}}
\providecommand{\U}[1]{\protect\rule{.1in}{.1in}}
\providecommand{\U}[1]{\protect\rule{.1in}{.1in}}
\providecommand{\U}[1]{\protect\rule{.1in}{.1in}}
\providecommand{\U}[1]{\protect\rule{.1in}{.1in}}
\providecommand{\U}[1]{\protect\rule{.1in}{.1in}}
\providecommand{\U}[1]{\protect\rule{.1in}{.1in}}
\providecommand{\U}[1]{\protect\rule{.1in}{.1in}}
\providecommand{\U}[1]{\protect\rule{.1in}{.1in}}
\providecommand{\U}[1]{\protect\rule{.1in}{.1in}}
\providecommand{\U}[1]{\protect\rule{.1in}{.1in}}
\providecommand{\U}[1]{\protect\rule{.1in}{.1in}}
\providecommand{\U}[1]{\protect\rule{.1in}{.1in}}
\providecommand{\U}[1]{\protect\rule{.1in}{.1in}}
\providecommand{\U}[1]{\protect\rule{.1in}{.1in}}
\providecommand{\U}[1]{\protect\rule{.1in}{.1in}}
\providecommand{\U}[1]{\protect\rule{.1in}{.1in}}
\providecommand{\U}[1]{\protect\rule{.1in}{.1in}}
\providecommand{\U}[1]{\protect\rule{.1in}{.1in}}

\newtheorem{theorem}{Theorem}
{}

{}

\newtheorem{remark}{Remark}

\newenvironment{proof}[1][Proof]{\textbf{#1.} }{\ \rule{0.5em}{0.5em}}

\begin{document}

\title{On the Basis Property of the Root Functions of Some Class of Non-self-adjoint
Sturm--Liouville Operators.}
\author{Cemile Nur\\{\small Depart. of Math., Dogus University, Ac\i badem, Kadik\"{o}y, \ }\\{\small Istanbul, Turkey.}\ {\small e-mail: cnur@dogus.edu.tr}
\and O. A. Veliev\\{\small Depart. of Math., Dogus University, Ac\i badem, Kadik\"{o}y, \ }\\{\small Istanbul, Turkey.}\ {\small e-mail: oveliev@dogus.edu.tr}}
\date{}
\maketitle

\begin{abstract}
We obtain the asymptotic formulas for the eigenvalues and eigenfunctions of
the Sturm-Liouville operators with some regular boundary conditions. Using
these formulas, we find sufficient conditions on the potential $q$ such that
the root functions of these operators do not form a Riesz basis.



\end{abstract}

\section{Introduction and Preliminary Facts}

Let $T_{1},T_{2},T_{3}$ and $T_{4}$ be the operators generated in $L_{2}[0,1]$
by the differential expression
\begin{equation}
l\left(  y\right)  =-y^{\prime\prime}+q(x)y
\end{equation}
and the following boundary conditions:
\begin{equation}
y_{0}^{\prime}+\beta y_{1}^{\prime}=0,\text{ }y_{0}-y_{1}=0,
\end{equation}%
\begin{equation}
y_{0}^{\prime}+\beta y_{1}^{\prime}=0,\text{ }y_{0}+y_{1}=0,
\end{equation}%
\begin{equation}
y_{0}^{\prime}-y_{1}^{\prime}=0,\text{ }y_{0}+\alpha y_{1}=0,
\end{equation}
and
\begin{equation}
y_{0}^{\prime}+y_{1}^{\prime}=0,\text{ }y_{0}+\alpha y_{1}=0
\end{equation}
respectively, where $q(x)$ is a complex-valued summable function on $[0,1]$,
$\beta\neq\pm1$ and $\alpha\neq\pm1.$

In conditions (2), (3), (4) and (5) if $\beta=1,$ $\beta=-1,$ $\alpha=1$ and
$\alpha=-1$ respectively, then any $\lambda\in%
\mathbb{C}
$ is an eigenvalue of infinite multiplicity. In (2) and (4) if $\beta=-1$ and
$\alpha=-1$ then they are periodic boundary conditions; In (3) and (5) if
$\beta=1$ and $\alpha=1$ then they are antiperiodic boundary conditions.

These boundary conditions are regular but not strongly regular. Note that, if
the boundary conditions are strongly regular, then the root functions form a
Riesz basis (this result was proved independently in [6], [10] and [17]). In
the case when an operator is regular but not strongly regular, the root
functions generally do not form even usual basis. However, Shkalikov [20],
[21] proved that they can be combined in pairs, so that the corresponding
2-dimensional subspaces form a Riesz basis of subspaces.

In the regular but not strongly regular boundary conditions, periodic and
antiperiodic boundary conditions are the ones more commonly studied.
Therefore, let us briefly describe some historical developments related to the
Riesz basis property of the root functions of the periodic and antiperiodic
boundary value problems. First results were obtained by Kerimov and Mamedov
[8]. They established that, if
\[
q\in C^{4}[0,1],\ q(1)\neq q(0),
\]
then the root functions of the operator $L_{0}(q)$ form a Riesz basis in
$L_{2}[0,1],$ where $L_{0}(q)$ denotes the operator generated by (1) and the
periodic boundary conditions.

The first result in terms of the Fourier coefficients of the potential $q$ was
obtained by Dernek and Veliev [1]. They proved that if the conditions
\begin{align}
\lim_{n\rightarrow\infty}\frac{\ln\left\vert n\right\vert }{nq_{2n}}  &
=0,\text{ }\\
q_{2n}  &  \sim q_{-2n}%
\end{align}
hold, then the root functions of $L_{0}(q)$ form a Riesz basis in $L_{2}%
[0,1]$, where $q_{n}=:(q,e^{i2\pi nx})$ is the Fourier coefficient of $q$ and
everywhere, without loss of generality, it is assumed that $q_{0}=0.$ Here
$(.,.)$ denotes the inner product in $L_{2}[0,1]$ and $a_{n}\sim b_{n}$ means
that $a_{n}=O(b_{n})$ and $b_{n}=O(a_{n})$ as $\ n\rightarrow\infty.$ Makin
[11] improved this result. Using another method he proved that the assertion
on the Riesz basis property remains valid if condition (7) holds, but
condition (6) is replaced by a less restrictive one: $q\in W_{1}^{s}[0,1],$
\[
q^{(k)}(0)=q^{(k)}(1),\quad\forall\,k=0,1,...,s-1
\]
holds and $\mid q_{2n}\mid>cn^{-s-1}$ with some$\ \,c>0$ for sufficiently
large $n,$ where $s$ is a nonnegative integer. Besides, some conditions which
imply the absence of the Riesz basis property were presented in [11].
Shkalilov and Veliev obtained in [22] more general results which cover all
results discussed above.

The other interesting results about periodic and antiperiodic boundary
conditions were obtained in [2-5, 7, 14-16, 24, 25].

The basis properties of regular but not strongly regular other some problems
are studied in [9,12,13]. It was proved in [12] that the system of the root
functions of the operator generated by (1) and the boundary conditions
\begin{align*}
y^{\prime}\left(  1\right)  -\left(  -1\right)  ^{\sigma}y^{\prime}\left(
0\right)  +\gamma y\left(  0\right)   &  =0\\
y\left(  1\right)  -\left(  -1\right)  ^{\sigma}y\left(  0\right)   &  =0,
\end{align*}
forms an unconditional basis of the space $L_{2}[0,1]$, where $q\left(
x\right)  $ is an arbitrary complex-valued function from the class
$L_{1}[0,1]$, $\gamma$ is an arbitrary nonzero complex constant and
$\sigma=0,1$. Kerimov and Kaya proved [9] that the system of the root
functions of the spectral problem
\begin{align*}
y^{\left(  4\right)  }+p_{2}\left(  x\right)  y^{\prime\prime}+p_{1}\left(
x\right)  y^{\prime}+p_{0}\left(  x\right)  y  &  =\lambda y,\text{ }0<x<1,\\
y^{\left(  s\right)  }\left(  1\right)  -\left(  -1\right)  ^{\sigma
}y^{\left(  s\right)  }\left(  0\right)  +\sum_{l=0}^{s-1}\alpha
_{s,l}y^{\left(  l\right)  }\left(  0\right)   &  =0,\text{ }s=1,2,3,\\
y\left(  1\right)  -\left(  -1\right)  ^{\sigma}y\left(  0\right)   &  =0,
\end{align*}
forms a basis in the space $L_{p}\left(  0,1\right)  $, $1<p<\infty$, when
$\alpha_{3,2}+\alpha_{1,0}\neq\alpha_{2,1}$, $p_{j}\left(  x\right)  \in
W_{1}^{j}\left(  0,1\right)  $, $j=1,2$, and $p_{0}\left(  x\right)  \in
L_{1}\left(  0,1\right)  $; moreover, this basis is unconditional for $p=2$,
where $\lambda$ is a spectral parameter; $p_{j}\left(  x\right)  \in
L_{1}\left(  0,1\right)  $, $j=1,2,3$, are complex-valued functions;
$\alpha_{s,l}$, $s=1,2,3$, $l=\overline{0,s-1}$ are arbitrary complex
constants; and $\sigma=0,1$.

It was shown in [19] that if
\[
q\left(  x\right)  =q\left(  1-x\right)  ,\text{ }\forall x\in\left[
0,1\right]  ,
\]
then the spectrum of each of the problems $T_{1}$, and $T_{3}$, coincides with
the spectrum of the periodic problem and the spectrum of each of the problems
$T_{2},$ and $T_{4}$, coincides with the spectrum of the antiperiodic problem.

In this paper we prove that if
\begin{equation}
\lim_{n\rightarrow\infty}\dfrac{\ln\left\vert n\right\vert }{ns_{2n}}=0,
\end{equation}
where $s_{k}=\left(  q,\sin2\pi kx\right)  ,$ then the large eigenvalues of
the operators $T_{1}$ and $T_{3}$ are simple. Moreover, if there exists a
sequence $\left\{  n_{k}\right\}  $ such that (8) holds when $n$ is replaced
by $n_{k},$ then the root functions of these operators do not form a Riesz basis.

Similarly, if
\begin{equation}
\lim_{n\rightarrow\infty}\dfrac{\ln\left\vert n\right\vert }{ns_{2n+1}}=0,
\end{equation}
then the large eigenvalues of the operators $T_{2}$ and $T_{4}$ are simple and
if there exists a sequence $\left\{  n_{k}\right\}  $ such that (9) holds when
$n$ is replaced by $n_{k},$ then the root functions of these operators do not
form a Riesz basis.

Moreover we obtain asymptotic formulas of arbitrary order for the eigenvalues
and eigenfunctions of the operators $T_{1}$,$T_{2},T_{3}$ and $T_{4}$.

\section{Main Results}

We will focus only on the operator $T_{1}$. The investigations of the
operators $T_{2},T_{3}$ and $T_{4}$ are similar. It is well-known that ( see
formulas (47a), (47b)) in page 65 of [18] ) the eigenvalues of the operators
$T_{1}(q)$ consist of the sequences $\{\lambda_{n,1}\},\{\lambda_{n,2}\}$
satisfying
\begin{equation}
\lambda_{n,j}=(2n\pi)^{2}+O(n^{1/2})
\end{equation}
for $j=1,2$. From this formula one can easily obtain the following inequality
\begin{equation}
\left\vert \lambda_{n,j}-(2\pi k)^{2}\right\vert =\left\vert 2(n-k)\pi
\right\vert \left\vert 2(n+k)\pi\right\vert +O(n^{\frac{1}{2}})>n
\end{equation}
for $j=1,2;$ $k\neq n;$ $k=0,1,...;$ and $n\geq N,$ where we denote by $N$ a
sufficiently large positive integer, that is, $N\gg1.$

It is easy to verify that if $q(x)=0$ then the eigenvalues of the operator
$T_{1},$ denoted by $T_{1}(0),$ are $\lambda_{n}=\left(  2\pi n\right)  ^{2}$
for $n=0,1,\ldots$ The eigenvalue $0$ is simple and the corresponding
eigenfunction is $1.$ The eigenvalues $\lambda_{n}=\left(  2\pi n\right)
^{2}$ for $n=1,2,\ldots$ are double and the corresponding eigenfunctions and
associated functions are
\begin{equation}
y_{n}\left(  x\right)  =\cos2\pi nx\text{ }\And\text{ }\phi_{n}\left(
x\right)  =\left(  \frac{\beta}{1+\beta}-x\right)  \frac{\sin2\pi nx}{4\pi n}%
\end{equation}
respectively. Note that for any constant $c$, $\phi_{n}\left(  x\right)
+cy_{n}\left(  x\right)  $ is also an associated function. It can be shown
that the adjoint operator $T_{1}^{\ast}(0)$ is associated with the boundary
conditions:
\begin{equation}
y_{1}+\overline{\beta}y_{0}=0,\text{ }y_{1}^{\prime}-y_{0}^{\prime}=0.
\end{equation}
It is easy to see that, $0$ is a simple eigenvalue of $T_{1}^{\ast}(0)$ and
the corresponding eigenfunction is $y_{0}^{\ast}\left(  x\right)  =x-\dfrac
{1}{1+\overline{\beta}}$ . The other eigenvalues $\lambda_{n}^{\ast}=\left(
2\pi n\right)  ^{2}$ for $n=1,2,\ldots$, are double and the corresponding
eigenfunctions and associated functions are
\begin{equation}
y_{n}^{\ast}\left(  x\right)  =\sin2\pi nx\text{ }\And\text{ }\phi_{n}^{\ast
}\left(  x\right)  =\left(  x-\dfrac{1}{1+\overline{\beta}}\right)  \frac
{\cos2\pi nx}{4\pi n}\nonumber
\end{equation}
respectively.

Let
\begin{equation}
\varphi_{n}\left(  x\right)  :=\frac{16\pi n\left(  \beta+1\right)  }{\beta
-1}\phi_{n}\left(  x\right)  =\frac{4\left(  \beta+1\right)  }{\beta-1}\left(
\dfrac{\beta}{1+\beta}-x\right)  \sin2\pi nx
\end{equation}
and
\begin{equation}
\varphi_{n}^{\ast}\left(  x\right)  :=\frac{16\pi n\left(  \overline{\beta
}+1\right)  }{\overline{\beta}-1}\phi_{n}^{\ast}\left(  x\right)
=\frac{4\left(  \overline{\beta}+1\right)  }{\overline{\beta}-1}\left(
x-\dfrac{1}{1+\overline{\beta}}\right)  \cos2\pi nx.
\end{equation}
The system of the root functions of $T_{1}^{\ast}(0)$ can be written as
$\{f_{n}:n\in\mathbb{Z}\},$ where
\begin{equation}
f_{-n}=\sin2\pi nx,\text{ }\forall n>0\And\text{ }f_{n}=\varphi_{n}^{\ast
}\left(  x\right)  ,\text{ }\forall n\geq0.
\end{equation}
One can easily verify that it forms a basis in $L_{2}[0,1]$ and the
biorthogonal system $\{g_{n}:n\in\mathbb{Z}\}$ is the system of the root
functions of $T_{1}(0),$ where
\begin{equation}
g_{-n}=\varphi_{n}\left(  x\right)  ,\forall n>0\text{ }\And\text{ }g_{n}%
=\cos2\pi nx,\forall n\geq0,
\end{equation}
since $\left(  f_{n},g_{m}\right)  =\delta_{n,m}.$

To obtain the asymptotic formulas for the eigenvalues $\lambda_{n,j}$ and the
corresponding normalized eigenfunctions $\Psi_{n,j}(x)$ of $T_{1}(q)$ we use
(11) and the well-known relations
\begin{equation}
(\lambda_{N,j}-(2\pi n)^{2})(\Psi_{N,j},\sin2\pi nx)=(q\Psi_{N,j},\sin2\pi nx)
\end{equation}
and%
\begin{equation}
\left(  \lambda_{N,j}-\left(  2\pi n\right)  ^{2}\right)  \left(  \Psi
_{N,j},\varphi_{n}^{\ast}\right)  -\gamma_{1}n\left(  \Psi_{N,j},\sin2\pi
nx\right)  =\left(  q\Psi_{N,j},\varphi_{n}^{\ast}\right)  ,
\end{equation}
where%
\[
\gamma_{1}=\frac{16\pi\left(  \beta+1\right)  }{\beta-1},
\]
which can be obtained by multiplying both sides of the equality
\[
-\left(  \Psi_{N,j}\right)  ^{\prime\prime}+q\left(  x\right)  \Psi
_{N,j}=\lambda_{N,j}\Psi_{N,j}%
\]
by $\sin2\pi nx$ and $\varphi_{n}^{\ast}$ respectively. It follows from (18)
and (19) that
\begin{equation}
\left(  \Psi_{N,j},\sin2\pi nx\right)  =\frac{\left(  q\left(  x\right)
\Psi_{N,j},\sin2\pi nx\right)  }{\lambda_{N,j}-\left(  2\pi n\right)  ^{2}%
};\text{ }N\neq n,
\end{equation}

\begin{equation}
\left(  \Psi_{N,j},\varphi_{n}^{\ast}\right)  =\frac{\gamma_{1}n\left(
q\left(  x\right)  \Psi_{N,j},\sin2\pi nx\right)  }{\left(  \lambda
_{N,j}-\left(  2\pi n\right)  ^{2}\right)  ^{2}}+\frac{\left(  q\left(
x\right)  \Psi_{N,j},\varphi_{n}^{\ast}\right)  }{\lambda_{N,j}-\left(  2\pi
n\right)  ^{2}};\text{ }N\neq n.
\end{equation}
Moreover, we use the following relations
\begin{align}
\left(  \Psi_{N,j},\overline{q}\sin2\pi nx\right)   &  =\sum_{n_{1}=0}%
^{\infty}[\left(  q\varphi_{n_{1}},\sin2\pi nx\right)  \left(  \Psi_{N,j}%
,\sin2\pi n_{1}x\right)  +\\
&  +\left(  q\cos2\pi n_{1}x,\sin2\pi nx\right)  \left(  \Psi_{N,j}%
,\varphi_{n_{1}}^{\ast}\right)  ],\nonumber
\end{align}%
\begin{equation}
\left(  \Psi_{N,j},\overline{q}\varphi_{n}^{\ast}\right)  =\sum_{n_{1}%
=0}^{\infty}\left[  \left(  q\varphi_{n_{1}},\varphi_{n}^{\ast}\right)
\left(  \Psi_{N,j},\sin2\pi n_{1}x\right)  +\left(  q\cos2\pi n_{1}%
x,\varphi_{n}^{\ast}\right)  \left(  \Psi_{N,j},\varphi_{n_{1}}^{\ast}\right)
\right]  ,
\end{equation}%
\begin{align}
\left\vert (q\Psi_{N,j},\sin2\pi nx)\right\vert  &  <4M,\\
\left\vert (q\Psi_{N,j},\varphi_{n}^{\ast})\right\vert  &  <4M,
\end{align}
for $N\gg1,$where $M=\sup\left\vert q_{n}\right\vert .$ These relations are
obvious for $q\in L_{2}(0,1),$ since to obtain (22) and (23) we can use the
decomposition of $\overline{q}\sin2\pi nx$ and $\overline{q}\varphi_{n}^{\ast
}$ by basis (16). For $q\in L_{1}(0,1)$\ see Lemma 1 of [23].

To obtain the asymptotic formulas for the eigenvalues and eigenfunctions we
iterate (18) and (19) by using (22), (23). First let us prove the following
obvious asymptotic formulas for the eigenfunctions $\Psi_{n,j}$. The expansion
of $\Psi_{n,j}$\ by basis (17) can be written in the form
\begin{equation}
\Psi_{n,j}=u_{n,j}\varphi_{n}\left(  x\right)  +v_{n,j}\cos2\pi nx+h_{n,j}%
\left(  x\right)  ,
\end{equation}
where
\begin{equation}
u_{n,j}=\left(  \Psi_{n,j},\sin2\pi nx\right)  ,\text{ }v_{n,j}=\left(
\Psi_{n,j},\varphi_{n}^{\ast}\right)  ,
\end{equation}%
\[
h_{n,j}\left(  x\right)  =\sum_{\substack{k=0\\k\neq n}}^{\infty}\left[
\left(  \Psi_{n,j},\sin2\pi kx\right)  \varphi_{k}\left(  x\right)  +\left(
\Psi_{n,j},\varphi_{k}^{\ast}\right)  \cos2\pi kx\right]  .
\]
Using (20), (21), (24) and (25) one can readily see that, there exists a
constant $C$ such that
\begin{equation}
\sup\left\vert h_{n,j}\left(  x\right)  \right\vert \leq C\left(  \sum_{k\neq
n}\left(  \frac{1}{\mid\lambda_{n,j}-\left(  2\pi k\right)  ^{2}\mid}+\frac
{n}{\left\vert \left(  \lambda_{n,j}-\left(  2\pi k\right)  ^{2}\right)
^{2}\right\vert }\right)  \right)  =O\left(  \frac{\ln n}{n}\right)
\end{equation}
and by (26) we get
\begin{equation}
\Psi_{n,j}=u_{n,j}\varphi_{n}\left(  x\right)  +v_{n,j}\cos2\pi nx+O\left(
\frac{\ln n}{n}\right)  .
\end{equation}

Since $\Psi_{n,j}$\ is normalized, we have
\[
1=\left\Vert \Psi_{n,j}\right\Vert ^{2}=\left(  \Psi_{n,j},\Psi_{n,j}\right)
=\left\vert u_{n,j}\right\vert ^{2}\left\Vert \varphi_{n}\left(  x\right)
\right\Vert ^{2}+\left\vert v_{n,j}\right\vert ^{2}\left\Vert \cos2\pi
nx\right\Vert ^{2}+
\]%
\[
+u_{n,j}\overline{v_{n,j}}\left(  \varphi_{n}\left(  x\right)  ,\cos2\pi
nx\right)  +v_{n,j}\overline{u_{n,j}}\left(  \cos2\pi nx,\varphi_{n}\left(
x\right)  \right)  +O\left(  \frac{\ln n}{n}\right)
\]%
\[
=\left(  \frac{8}{3}\dfrac{\left\vert \beta\right\vert ^{2}-\operatorname{Re}%
\beta+1}{\left\vert \beta-1\right\vert ^{2}}\right)  \left\vert u_{n,j}%
\right\vert ^{2}+\frac{1}{2}\left\vert v_{n,j}\right\vert ^{2}+O\left(
\frac{\ln n}{n}\right)  ,
\]
that is,%
\begin{equation}
a\left\vert u_{n,j}\right\vert ^{2}+\frac{1}{2}\left\vert v_{n,j}\right\vert
^{2}=1+O\left(  \frac{\ln n}{n}\right)  ,
\end{equation}
where
\[
a=\frac{8}{3}\dfrac{\left\vert \beta\right\vert ^{2}-\operatorname{Re}\beta
+1}{\left\vert \beta-1\right\vert ^{2}}.
\]
Note that $a\neq0$, since $\left\vert \beta\right\vert ^{2}+1>\left\vert
\beta\right\vert .$

Now let us iterate (18). Using (22) in (18) we get
\begin{gather*}
\left(  \lambda_{n,j}-\left(  2\pi n\right)  ^{2}\right)  \left(  \Psi
_{n,j},\sin2\pi nx\right)  =\\
=\sum_{n_{1}=0}^{\infty}\left[  \left(  q\varphi_{n_{1}},\sin2\pi nx\right)
\left(  \Psi_{n,j},\sin2\pi n_{1}x\right)  +\left(  q\cos2\pi n_{1}x,\sin2\pi
nx\right)  \left(  \Psi_{n,j},\varphi_{n_{1}}^{\ast}\left(  x\right)  \right)
\right]  .
\end{gather*}
Isolating the terms in the right-hand side of this equality containing the
multiplicands $\left(  \Psi_{n,j},\sin2\pi nx\right)  $ and $\left(
\Psi_{n,j},\varphi_{n}^{\ast}\left(  x\right)  \right)  $ (i.e., case
$n_{1}=n$ ), using\ (20) and (21) for the terms $\left(  \Psi_{n,j},\sin2\pi
n_{1}x\right)  $ and \ $\left(  \Psi_{n,j},\varphi_{n_{1}}^{\ast}\left(
x\right)  \right)  $ respectively (in the case $n_{1}\neq n$) we obtain%

\begin{gather*}
\left[  \lambda_{n,j}-\left(  2\pi n\right)  ^{2}-\left(  q\varphi_{n}%
,\sin2\pi nx\right)  \right]  \left(  \Psi_{n,j},\sin2\pi nx\right)  -\left(
q\cos2\pi nx,\sin2\pi nx\right)  \left(  \Psi_{n,j},\varphi_{n}^{\ast}\right)
=\\
=\sum_{\substack{n_{1}=0\\n_{1}\neq n}}^{\infty}\left[  \left(  q\varphi
_{n_{1}},\sin2\pi nx\right)  \left(  \Psi_{n,j},\sin2\pi n_{1}x\right)
+\left(  q\cos2\pi n_{1}x,\sin2\pi nx\right)  \left(  \Psi_{n,j}%
,\varphi_{n_{1}}^{\ast}\left(  x\right)  \right)  \right] \\
=\sum_{n_{1}}\left[  a_{1}\left(  \lambda_{n,j}\right)  \left(  q\left(
x\right)  \Psi_{n,j},\sin2\pi n_{1}x\right)  +b_{1}\left(  \lambda
_{n,j}\right)  \left(  q\left(  x\right)  \Psi_{n,j},\varphi_{n_{1}}^{\ast
}\right)  \right]  .
\end{gather*}
where
\begin{align*}
a_{1}\left(  \lambda_{n,j}\right)   &  =\frac{\left(  q\varphi_{n_{1}}%
,\sin2\pi nx\right)  }{\lambda_{n,j}-\left(  2\pi n_{1}\right)  ^{2}}%
+\frac{\gamma_{1}n_{1}\left(  q\cos2\pi n_{1}x,\sin2\pi nx\right)  }{\left(
\lambda_{n,j}-\left(  2\pi n_{1}\right)  ^{2}\right)  ^{2}},\\
b_{1}\left(  \lambda_{n,j}\right)   &  =\frac{\left(  q\cos2\pi n_{1}%
x,\sin2\pi nx\right)  }{\lambda_{n,j}-\left(  2\pi n_{1}\right)  ^{2}}.
\end{align*}

Using (22) and (23) for the terms $\left(  q\Psi_{n,j},\sin2\pi n_{1}x\right)
$ and $\left(  q\Psi_{n,j},\varphi_{n_{1}}^{\ast}\right)  $ of the last
summation we obtain%

\begin{gather*}
\left[  \lambda_{n,j}-\left(  2\pi n\right)  ^{2}-\left(  q\varphi_{n}%
,\sin2\pi nx\right)  \right]  \left(  \Psi_{n,j},\sin2\pi nx\right)  -\left(
q\cos2\pi nx,\sin2\pi nx\right)  \left(  \Psi_{n,j},\varphi_{n}^{\ast}\right)
=\\
=\sum_{n_{1}}\left[  a_{1}\left(  \lambda_{n,j}\right)  \left(  q\Psi
_{n,j},\sin2\pi n_{1}x\right)  +b_{1}\left(  \lambda_{n,j}\right)  \left(
q\Psi_{n,j},\varphi_{n_{1}}^{\ast}\right)  \right]  =\\
=\sum_{n_{1}}a_{1}\left(  \sum_{n_{2}=0}^{\infty}\left[  \left(
q\varphi_{n_{2}},\sin2\pi n_{1}x\right)  \left(  \Psi_{n,j},\sin2\pi
n_{2}x\right)  +\left(  q\cos2\pi n_{2}x,\sin2\pi n_{1}x\right)  \left(
\Psi_{n,j},\varphi_{n_{2}}^{\ast}\left(  x\right)  \right)  \right]  \right)
+\\
+\sum_{n_{1}}b_{1}\left(  \sum_{n_{2}=0}^{\infty}\left[  \left(
q\varphi_{n_{2}},\varphi_{n_{1}}^{\ast}\right)  \left(  \Psi_{n,j},\sin2\pi
n_{2}x\right)  +\left(  q\cos2\pi n_{2}x,\varphi_{n_{1}}^{\ast}\right)
\left(  \Psi_{n,j},\varphi_{n_{2}}^{\ast}\left(  x\right)  \right)  \right]
\right)  .
\end{gather*}
Now isolating the terms for $n_{2}=n$ we get%
\begin{gather*}
\left[  \lambda_{n,j}-\left(  2\pi n\right)  ^{2}-\left(  q\varphi_{n}%
,\sin2\pi nx\right)  \right]  \left(  \Psi_{n,j},\sin2\pi nx\right)  -\left(
q\cos2\pi nx,\sin2\pi nx\right)  \left(  \Psi_{n,j},\varphi_{n}^{\ast}\right)
=\\
=\sum_{n_{1}}\left[  a_{1}\left(  q\varphi_{n},\sin2\pi n_{1}x\right)
+b_{1}\left(  q\varphi_{n},\varphi_{n_{1}}^{\ast}\right)  \right]  \left(
\Psi_{n,j},\sin2\pi nx\right)  +\\
+\sum_{n_{1}}\left[  a_{1}\left(  q\cos2\pi nx,\sin2\pi n_{1}x\right)
+b_{1}\left(  q\cos2\pi nx,\varphi_{n_{1}}^{\ast}\right)  \right]  \left(
\Psi_{n,j},\varphi_{n}^{\ast}\left(  x\right)  \right)  +\\
=\sum_{n_{1},n_{2}}\left(  \left[  a_{1}\left(  q\varphi_{n_{2}},\sin2\pi
n_{1}x\right)  +b_{1}\left(  q\varphi_{n_{2}},\varphi_{n_{1}}^{\ast}\right)
\right]  \left(  \Psi_{n,j},\sin2\pi n_{2}x\right)  +\right)  +\\
+\sum_{n_{1},n_{2}}\left[  a_{1}\left(  q\cos2\pi n_{2}x,\sin2\pi
n_{1}x\right)  +b_{1}\left(  q\cos2\pi n_{2}x,\varphi_{n_{1}}^{\ast}\right)
\right]  \left(  \Psi_{n,j},\varphi_{n_{2}}^{\ast}\right)  .
\end{gather*}
Here and further the summations are taken under the conditions $n_{i}\neq n$
and $n_{i}=0,1,...$ for $i=1,2,...$ Introduce the notations%
\begin{align*}
C_{1} &  =:a_{1},\text{ }M_{1}=:b_{1},\\
C_{2} &  =:a_{1}a_{2}+b_{1}A_{2}=C_{1}a_{2}+M_{1}A_{2},\text{ }M_{2}%
=:a_{1}b_{2}+b_{1}B_{2}=C_{1}b_{2}+M_{1}B_{2},\\
C_{k+1} &  =:C_{k}a_{k+1}+M_{k}A_{k+1},\text{ }M_{k+1}=:C_{k}b_{k+1}%
+M_{k}B_{k+1};\text{ }k=2,3,\ldots,
\end{align*}

where%
\begin{gather*}
a_{k+1}=a_{k+1}\left(  \lambda_{n,j}\right)  =\dfrac{\left(  q\varphi
_{n_{k+1}},\sin2\pi n_{k}x\right)  }{\lambda_{n,j}-\left(  2\pi n_{k+1}%
\right)  ^{2}}+\dfrac{\gamma_{1}n_{k+1}\left(  q\cos2\pi n_{k+1}x,\sin2\pi
n_{k}x\right)  }{\left(  \lambda_{n,j}-\left(  2\pi n_{k+1}\right)
^{2}\right)  ^{2}},\\
b_{k+1}=b_{k+1}\left(  \lambda_{n,j}\right)  =\dfrac{\left(  q\cos2\pi
n_{k+1}x,\sin2\pi n_{k}x\right)  }{\lambda_{n,j}-\left(  2\pi n_{k+1}\right)
^{2}},\\
A_{k+1}=A_{k+1}\left(  \lambda_{n,j}\right)  =\dfrac{\left(  q\varphi
_{n_{k+1}},\varphi_{n_{k}}^{\ast}\right)  }{\lambda_{n,j}-\left(  2\pi
n_{k+1}\right)  ^{2}}+\dfrac{\gamma_{1}n_{k+1}\left(  q\cos2\pi n_{k+1}%
x,\varphi_{n_{k}}^{\ast}\right)  }{\left(  \lambda_{n,j}-\left(  2\pi
n_{k+1}\right)  ^{2}\right)  ^{2}},\\
B_{k+1}=B_{k+1}\left(  \lambda_{n,j}\right)  =\dfrac{\left(  q\cos2\pi
n_{k+1}x,\varphi_{n_{k}}^{\ast}\right)  }{\lambda_{n,j}-\left(  2\pi
n_{k+1}\right)  ^{2}}.
\end{gather*}
Using these notations and repeating this iteration $k$ times we get
\begin{gather}
\left[  \lambda_{n,j}-\left(  2\pi n\right)  ^{2}-\left(  q\varphi_{n}%
,\sin2\pi nx\right)  -\widetilde{A}_{k}\left(  \lambda_{n,j}\right)  \right]
\left(  \Psi_{n,j},\sin2\pi nx\right)  =\nonumber\\
=\left[  \left(  q\cos2\pi nx,\sin2\pi nx\right)  +\widetilde{B}_{k}\left(
\lambda_{n,j}\right)  \right]  \left(  \Psi_{n,j},\varphi_{n}^{\ast}\left(
x\right)  \right)  +R_{k},
\end{gather}
where
\begin{align*}
\widetilde{A}_{k}\left(  \lambda_{n,j}\right)   &  =\sum_{m=1}^{k}\alpha
_{m}\left(  \lambda_{n,j}\right)  \text{, }\widetilde{B}_{k}\left(
\lambda_{n,j}\right)  =\sum_{m=1}^{k}\beta_{m}\left(  \lambda_{n,j}\right)
,\\
\alpha_{k}\left(  \lambda_{n,j}\right)   &  =\sum_{n_{1},\ldots,n_{k}}\left[
C_{k}\left(  q\varphi_{n},\sin2\pi n_{k}x\right)  +M_{k}\left(  q\varphi
_{n},\varphi_{n_{k}}^{\ast}\right)  \right]  ,\\
\beta_{k}\left(  \lambda_{n,j}\right)   &  =\sum_{n_{1},\ldots,n_{k}}\left[
C_{k}\left(  q\cos2\pi nx,\sin2\pi n_{k}x\right)  +M_{k}\left(  q\cos2\pi
nx,\varphi_{n_{k}}^{\ast}\right)  \right]  ,\\
R_{k}  &  =\sum_{n_{1},\ldots,n_{k+1}}\left\{  C_{k+1}\left(  q\Psi_{n,j}%
,\sin2\pi n_{k+1}x\right)  +M_{k+1}\left(  q\Psi_{n,j},\varphi_{n_{k+1}}%
^{\ast}\right)  \right\}  .
\end{align*}
It follows from (11), (24) and (25) that
\begin{equation}
\alpha_{k}\left(  \lambda_{n,j}\right)  =O\left(  \left(  \frac{\ln\left\vert
n\right\vert }{n}\right)  ^{k}\right)  ,\beta_{k}\left(  \lambda_{n,j}\right)
=O\left(  \left(  \frac{\ln\left\vert n\right\vert }{n}\right)  ^{k}\right)
,R_{k}=O\left(  \left(  \frac{\ln\left\vert n\right\vert }{n}\right)
^{k+1}\right)  .
\end{equation}

Therefore if we take limit in (31) for $k\rightarrow\infty$, we obtain
\[
\left[  \lambda_{n,j}-\left(  2\pi n\right)  ^{2}-Q_{n}-A\left(  \lambda
_{n,j}\right)  \right]  u_{n,j}=\left[  P_{n}+B\left(  \lambda_{n,j}\right)
\right]  v_{n,j},
\]
where
\begin{equation}
P_{n}=\left(  q\cos2\pi nx,\sin2\pi nx\right)  ,\text{ }Q_{n}=\left(
q\varphi_{n},\sin2\pi nx\right)  ,
\end{equation}%
\begin{equation}
A\left(  \lambda_{n,j}\right)  =\sum_{m=1}^{\infty}\alpha_{m}\left(
\lambda_{n,j}\right)  =O\left(  \frac{\ln\left\vert n\right\vert }{n}\right)
\text{, }B\left(  \lambda_{n,j}\right)  =\sum_{m=1}^{\infty}\beta_{m}\left(
\lambda_{n,j}\right)  =O\left(  \frac{\ln\left\vert n\right\vert }{n}\right)
.
\end{equation}

Thus iterating (18) we obtained (31). Now starting \ to iteration from (19)
instead of (18) and using (23), (22) and arguing as in the previous iteration,
we get%
\begin{equation}
\left[  \lambda_{n,j}-\left(  2\pi n\right)  ^{2}-P_{n}^{\ast}-A_{k}^{\prime
}\left(  \lambda_{n,j}\right)  \right]  v_{n,j}=\left[  \gamma_{1}%
n+Q_{n}^{\ast}+B_{k}^{\prime}\left(  \lambda_{n,j}\right)  \right]
u_{n,j}+R_{k}^{\prime},
\end{equation}
where
\begin{equation}
P_{n}^{\ast}=\left(  q\cos2\pi nx,\varphi_{n}^{\ast}\right)  ,\text{ }%
Q_{n}^{\ast}=\left(  q\varphi_{n},\varphi_{n}^{\ast}\right)  ,
\end{equation}%
\begin{align*}
A_{k}^{\prime}\left(  \lambda_{n,j}\right)   &  =\sum_{m=1}^{k}\alpha
_{m}^{\prime}\left(  \lambda_{n,j}\right)  \text{, }B_{k}^{\prime}\left(
\lambda_{n,j}\right)  =\sum_{m=1}^{k}\beta_{m}^{\prime}\left(  \lambda
_{n,j}\right)  ,\\
\alpha_{k}^{\prime}\left(  \lambda_{n,j}\right)   &  =\sum_{n_{1},\ldots
,n_{k}}\left[  \widetilde{C}_{k}\left(  q\cos2\pi nx,\sin2\pi n_{k}x\right)
+\widetilde{M}_{k}\left(  q\cos2\pi nx,\varphi_{n_{k}}^{\ast}\right)  \right]
,\\
\beta_{k}^{\prime}\left(  \lambda_{n,j}\right)   &  =\sum_{n_{1},\ldots,n_{k}%
}\left[  \widetilde{C}_{k}\left(  q\varphi_{n},\sin2\pi n_{k}x\right)
+\widetilde{M}_{k}\left(  q\varphi_{n},\varphi_{n_{k}}^{\ast}\right)  \right]
,\\
R_{k}^{\prime}  &  =\sum_{n_{1},\ldots,n_{k+1}}\left\{  \widetilde{C}%
_{k+1}\left(  q\Psi_{n,j},\sin2\pi n_{k+1}x\right)  +\widetilde{M}%
_{k+1}\left(  q\Psi_{n,j},\varphi_{n_{k+1}}^{\ast}\right)  \right\}  ,
\end{align*}

\[
\widetilde{C}_{k+1}=\widetilde{C}_{k}a_{k+1}+\widetilde{M}_{k}A_{k+1},\text{
}\widetilde{M}_{k+1}=\widetilde{C}_{k}b_{k+1}+\widetilde{M}_{k}B_{k+1};\text{
}k=0,1,2,\ldots,
\]%
\begin{align*}
\widetilde{C}_{1} &  =A_{1}\left(  \lambda_{n,j}\right)  =\frac{\left(
q\varphi_{n_{1}},\varphi_{n}^{\ast}\right)  }{\lambda_{n,j}-\left(  2\pi
n_{1}\right)  ^{2}}+\frac{\gamma_{1}n_{1}\left(  q\cos2\pi n_{1}x,\varphi
_{n}^{\ast}\right)  }{\left(  \lambda_{n,j}-\left(  2\pi n_{1}\right)
^{2}\right)  ^{2}},\\
\widetilde{M}_{1} &  =B_{1}\left(  \lambda_{n,j}\right)  =\frac{\left(
q\cos2\pi n_{1}x,\varphi_{n}^{\ast}\right)  }{\lambda_{n,j}-\left(  2\pi
n_{1}\right)  ^{2}}.
\end{align*}
Similar to (32) one can verify that%
\begin{equation}
\alpha_{k}^{\prime}\left(  \lambda_{n,j}\right)  =O\left(  \left(  \frac
{\ln\left\vert n\right\vert }{n}\right)  ^{k}\right)  ,\beta_{k}^{\prime
}\left(  \lambda_{n,j}\right)  =O\left(  \left(  \frac{\ln\left\vert
n\right\vert }{n}\right)  ^{k}\right)  ,R_{k}^{\prime}=O\left(  \left(
\frac{\ln\left\vert n\right\vert }{n}\right)  ^{k+1}\right)  .
\end{equation}
If we take limit in (35) for $k\rightarrow\infty$, we obtain
\[
\left[  \lambda_{n,j}-\left(  2\pi n\right)  ^{2}-P_{n}^{\ast}-A^{\prime
}\left(  \lambda_{n,j}\right)  \right]  v_{n,j}=\left[  \gamma_{1}%
n+Q_{n}^{\ast}+B^{\prime}\left(  \lambda_{n,j}\right)  \right]  u_{n,j},
\]
where
\begin{equation}
A^{\prime}\left(  \lambda_{n,j}\right)  =\sum_{m=1}^{\infty}\alpha_{m}%
^{\prime}\left(  \lambda_{n,j}\right)  =O\left(  \frac{\ln\left\vert
n\right\vert }{n}\right)  \text{, }B^{\prime}\left(  \lambda_{n,j}\right)
=\sum_{m=1}^{\infty}\beta_{m}^{\prime}\left(  \lambda_{n,j}\right)  =O\left(
\frac{\ln\left\vert n\right\vert }{n}\right)  .
\end{equation}
To get the main results of this paper we use the following system of
equations, obtained above, with respect to $u_{n,j}$ and $v_{n,j}$
\begin{gather}
\left[  \lambda_{n,j}-\left(  2\pi n\right)  ^{2}-Q_{n}-A\left(  \lambda
_{n,j}\right)  \right]  u_{n,j}=\left[  P_{n}+B\left(  \lambda_{n,j}\right)
\right]  v_{n,j},\\
\left[  \lambda_{n,j}-\left(  2\pi n\right)  ^{2}-P_{n}^{\ast}-A^{\prime
}\left(  \lambda_{n,j}\right)  \right]  v_{n,j}=\left[  \gamma_{1}%
n+Q_{n}^{\ast}+B^{\prime}\left(  \lambda_{n,j}\right)  \right]  u_{n,j},
\end{gather}
where
\begin{gather}
Q_{n}=\left(  q\varphi_{n},\sin2\pi nx\right)  =\nonumber\\
=-\frac{2\left(  \beta+1\right)  }{\beta-1}\int_{0}^{1}xq\left(  x\right)
dx+\frac{2\left(  \beta+1\right)  }{\beta-1}\left(  xq\left(  x\right)
,\cos4\pi nx\right)  -\frac{2\beta}{\beta-1}\left(  q\left(  x\right)
,\cos4\pi nx\right)  \\
=-\frac{2\left(  \beta+1\right)  }{\beta-1}\int_{0}^{1}xq\left(  x\right)
dx+o\left(  1\right)  ,
\end{gather}%
\begin{gather}
P_{n}^{\ast}=\left(  q\cos2\pi nx,\varphi_{n}^{\ast}\right)  =\nonumber\\
=\frac{2\left(  \beta+1\right)  }{\beta-1}\int_{0}^{1}xq\left(  x\right)
dx+\frac{2\left(  \beta+1\right)  }{\beta-1}\left(  xq\left(  x\right)
,\cos4\pi nx\right)  -\frac{2}{\beta-1}\left(  q\left(  x\right)  ,\cos4\pi
nx\right)  \\
=\frac{2\left(  \beta+1\right)  }{\beta-1}\int_{0}^{1}xq\left(  x\right)
dx+o\left(  1\right)  ,
\end{gather}%
\begin{equation}
P_{n}=\left(  q\cos2\pi nx,\sin2\pi nx\right)  =\frac{1}{2}\left(  q,\sin4\pi
nx\right)  =o\left(  1\right)  ,
\end{equation}%
\begin{equation}
Q_{n}^{\ast}=\left(  q\varphi_{n},\varphi_{n}^{\ast}\right)  =8\left(
\frac{\beta_{1}+1}{\beta_{1}-1}\right)  ^{2}\int_{0}^{1}q\left(  x\right)
\left(  \dfrac{\beta_{1}}{1+\beta_{1}}-x\right)  \left(  x-\dfrac{1}%
{1+\beta_{1}}\right)  \sin4\pi nxdx=o\left(  1\right)  .
\end{equation}
Note that (39), (40) with (34), (38) give us
\begin{gather}
\left[  \lambda_{n,j}-\left(  2\pi n\right)  ^{2}-Q_{n}+O\left(  \dfrac
{\ln\left\vert n\right\vert }{n}\right)  \right]  u_{n,j}=\left[
P_{n}+O\left(  \dfrac{\ln\left\vert n\right\vert }{n}\right)  \right]
v_{n,j},\\
\left[  \lambda_{n,j}-\left(  2\pi n\right)  ^{2}-P_{n}^{\ast}+O\left(
\dfrac{\ln\left\vert n\right\vert }{n}\right)  \right]  v_{n,j}=\left[
\gamma_{1}n+Q_{n}^{\ast}+O\left(  \dfrac{\ln\left\vert n\right\vert }%
{n}\right)  \right]  u_{n,j}.
\end{gather}

Introduce the notations
\begin{align}
c_{n}  &  =\left(  q,\cos2\pi nx\right)  \text{, }s_{n}=\left(  q,\sin2\pi
nx\right) \nonumber\\
c_{n,1}  &  =\left(  xq,\cos2\pi nx\right)  \text{, }s_{n,1}=\left(
xq,\sin2\pi nx\right) \\
c_{n,2}  &  =\left(  x^{2}q,\cos2\pi nx\right)  \text{, }s_{n,2}=\left(
x^{2}q,\sin2\pi nx\right)  .\nonumber
\end{align}
In these notations we have
\begin{equation}
Q_{n}=-\frac{2\left(  \beta+1\right)  }{\beta-1}\int_{0}^{1}xq\left(
x\right)  dx+\frac{2\left(  \beta+1\right)  }{\beta-1}c_{2n,1}-\frac{2\beta
}{\beta-1}c_{2n}%
\end{equation}%
\begin{equation}
P_{n}^{\ast}=\frac{2\left(  \beta+1\right)  }{\beta-1}\int_{0}^{1}xq\left(
x\right)  dx+\frac{2\left(  \beta+1\right)  }{\beta-1}c_{2n,1}-\frac{2}%
{\beta-1}c_{2n}%
\end{equation}%
\begin{equation}
P_{n}=\frac{1}{2}s_{2n}%
\end{equation}%
\begin{equation}
Q_{n}^{\ast}=-8\left(  \frac{\beta+1}{\beta-1}\right)  ^{2}s_{2n,2}+8\left(
\frac{\beta+1}{\beta-1}\right)  ^{2}s_{2n,1}-\frac{8\beta}{\left(
\beta-1\right)  ^{2}}s_{2n}.
\end{equation}

\begin{theorem}
For $j=1,2$ the following statements hold:

$(a)$ Any eigenfunction $\Psi_{n,j}$ of $T_{1}$ corresponding to the
eigenvalue $\lambda_{n,j}$ defined in (10) satisfies
\begin{equation}
\Psi_{n,j}=\sqrt{2}\cos2\pi nx+O\left(  n^{-1/2}\right)  .
\end{equation}
Moreover there exists $N$ such that for all $n>N$ the geometric multiplicity
of the eigenvalue $\lambda_{n,j}$ is $1$.

$\left(  b\right)  $ A complex number $\lambda\in U(n)=:\{\lambda:\left\vert
\lambda-\left(  2\pi n\right)  ^{2}\right\vert \leq n\}$ is an eigenvalue of
$T_{1}$ if and only if it is a root of the equation
\begin{gather}
\left[  \lambda-\left(  2\pi n\right)  ^{2}-Q_{n}-A\left(  \lambda\right)
\right]  \left[  \lambda-\left(  2\pi n\right)  ^{2}-P_{n}^{\ast}-A^{\prime
}\left(  \lambda\right)  \right]  -\nonumber\\
-\left[  P_{n}+B\left(  \lambda\right)  \right]  \left[  \gamma_{1}%
n+Q_{n}^{\ast}+B^{\prime}\left(  \lambda\right)  \right]  =0.
\end{gather}
Moreover $\lambda\in U(n)$ is a double eigenvalue of $T_{1}$ if and only if
\textit{it is a double root of} (55) .
\end{theorem}

\begin{proof}
$\left(  a\right)  $ By (10) the left hand side of (48) is $O(n^{1/2}),$ which
implies that $u_{n,j}=O(n^{-1/2}).$ Therefore from (29) we obtain (54). Now
suppose that there are two linearly independent eigenfunctions corresponding
to $\lambda_{n,j}$. Then there exists an eigenfunction satisfying
\[
\Psi_{n,j}=\sqrt{2}\sin2\pi nx+o\left(  1\right)
\]
which contradicts (54).

$(b)$ First we prove that the large eigenvalues $\lambda_{n,j}$ are the roots
of the equation (55). It follows from (54), (27) and (15) that $v_{n,j}\neq0.$
If $u_{n,j}\neq0$ then multiplying the equations (39) and (40) side by side
and then canceling $v_{n,j}u_{n,j}$ we obtain (55) . If $u_{n,j}=0$ then by
(39) and (40) we have $P_{n}+B\left(  \lambda_{n,j}\right)  =0$ \ and
$\lambda_{n,j}-\left(  2\pi n\right)  ^{2}-P_{n}^{\ast}-A^{\prime}\left(
\lambda_{n,j}\right)  =0$ which mean that (55) holds. Thus in any case
$\lambda_{n,j}$ is a root of (55).

Now we prove that the roots of (55) lying in $U(n)$\ are the eigenvalues of
$T_{1}.$ Let $F(\lambda)$ be the left-hand side of (55) which can be written
as
\begin{gather}
F(\lambda)=(\lambda-\left(  2\pi n\right)  ^{2})^{2}-\left(  Q_{n}+A\left(
\lambda\right)  +P_{n}^{\ast}+A^{\prime}\left(  \lambda\right)  \right)
\left(  \lambda-\left(  2\pi n\right)  ^{2}\right)  +\\
+\left(  Q_{n}+A\left(  \lambda\right)  \right)  \left(  P_{n}^{\ast
}+A^{\prime}\left(  \lambda\right)  \right)  -\left(  P_{n}+B\left(
\lambda\right)  \right)  \left(  \gamma_{1}n+Q_{n}^{\ast}+B^{\prime}\left(
\lambda\right)  \right) \nonumber
\end{gather}
and
\[
G(\lambda)=(\lambda-\left(  2\pi n\right)  ^{2})^{2}.
\]
One can easily verify that the inequality
\begin{equation}
\mid F(\lambda)-G(\lambda)\mid<\mid G(\lambda)\mid
\end{equation}
holds\ for all $\lambda$ from the boundary of $U(n).$ Since the function
$G(\lambda)$ has two roots in the set $U(n),$ by the Rouche's theorem we
obtain that $F(\lambda)$ has two roots in the same\ set.\ Thus\ $T_{1}$ has
two eigenvalues (counting with multiplicities) lying in $U(n)$ that are the
roots of (55). On the other hand, (55) has preciously two roots (counting with
multiplicities) in $U(n).$ Therefore $\lambda\in U(n)$ is an eigenvalue of
$T_{1}$ if and only if (55) holds.

If \textit{ }$\lambda\in U(n)$ is a double eigenvalue of $T_{1}$ then it has
no other eigenvalues\textit{ }in\textit{ }$U(n)$ and hence (55) has no other
roots. This implies that $\lambda$ is a double root of (55). By the same way
one can prove that if $\lambda$ is a double root of (55) then it is a double
eigenvalue of $T_{1}.$
\end{proof}

Let us consider (55) in detail. If we substitute $t=:\lambda-\left(  2\pi
n\right)  ^{2}$ then it becomes
\begin{gather}
t^{2}-\left(  Q_{n}+A\left(  \lambda\right)  +P_{n}^{\ast}+A^{\prime}\left(
\lambda\right)  \right)  t+\\
+\left(  Q_{n}+A\left(  \lambda\right)  \right)  \left(  P_{n}^{\ast
}+A^{\prime}\left(  \lambda\right)  \right)  -\left(  P_{n}+B\left(
\lambda\right)  \right)  \left(  \gamma_{1}n+Q_{n}^{\ast}+B^{\prime}\left(
\lambda\right)  \right)  =0.\nonumber
\end{gather}
The solutions of this equation are
\[
t_{1,2}=\frac{\left(  Q_{n}+P_{n}^{\ast}+A+A^{\prime}\right)  \pm\sqrt
{\Delta\left(  \lambda\right)  }}{2},
\]
where
\begin{equation}
\Delta\left(  \lambda\right)  =\left(  Q_{n}+P_{n}^{\ast}+A+A^{\prime}\right)
^{2}-4\left(  Q_{n}+A\right)  \left(  P_{n}^{\ast}+A^{\prime}\right)
+4\left(  P_{n}+B\right)  \left(  \gamma_{1}n+Q_{n}^{\ast}+B^{\prime}\right)
\nonumber
\end{equation}
which can be written in the form
\begin{equation}
\Delta\left(  \lambda\right)  =\left(  Q_{n}-P_{n}^{\ast}+A-A^{\prime}\right)
^{2}+4\left(  P_{n}+B\right)  \left(  \gamma_{1}n+Q_{n}^{\ast}+B^{\prime
}\right)  .
\end{equation}
Clearly the eigenvalue $\lambda_{n,j}$\ is a root either of the equation
\begin{equation}
\lambda=\left(  2\pi n\right)  ^{2}+\frac{1}{2}\left[  \left(  Q_{n}%
+P_{n}^{\ast}+A+A^{\prime}\right)  -\sqrt{\Delta\left(  \lambda\right)
}\right]
\end{equation}
or of the equation%
\begin{equation}
\lambda=\left(  2\pi n\right)  ^{2}+\frac{1}{2}\left[  \left(  Q_{n}%
+P_{n}^{\ast}+A+A^{\prime}\right)  +\sqrt{\Delta\left(  \lambda\right)
}\right]  .
\end{equation}
Now let us examine $\Delta\left(  \lambda_{n,j}\right)  $ in detail. If (8)
holds then one can readily see from (34), (38), (50), (51) and (59) that%
\begin{equation}
\Delta\left(  \lambda_{n,j}\right)  =2\gamma_{1}ns_{2n}(1+o(1)).
\end{equation}
Taking into account the last three equality and (34), (38), (50), (51), we see
that (60) and (61) have the form
\begin{equation}
\lambda=\left(  2\pi n\right)  ^{2}-\frac{\sqrt{2\gamma_{1}}}{2}\sqrt{ns_{2n}%
}(1+o(1)),
\end{equation}%
\begin{equation}
\lambda=\left(  2\pi n\right)  ^{2}+\frac{\sqrt{2\gamma_{1}}}{2}\sqrt{ns_{2n}%
}(1+o(1)).
\end{equation}

\begin{theorem}
If (8) holds, then the large eigenvalues $\lambda_{n,j}$ are simple and
satisfy the following asymptotic formulas%
\begin{equation}
\lambda_{n,j}=\left(  2\pi n\right)  ^{2}+\left(  -1\right)  ^{j}\frac
{\sqrt{2\gamma_{1}}}{2}\sqrt{ns_{2n}}(1+o(1)).
\end{equation}
for $j=1,2.$ Moreover, if there exists a sequence $\left\{  n_{k}\right\}  $
such that (8) holds when $n$ is replaced by $n_{k},$ then the root functions
of $T_{1}$ do not form a Riesz basis.
\end{theorem}

\begin{proof}
To prove that the large eigenvalues $\lambda_{n,j}$ are simple let us show
that one of the eigenvalues, say $\lambda_{n,1}$ satisfies (65) for $j=1$ and
the other $\lambda_{n,2}$ satisfies (65) for $j=2.$ Let us prove that each of
the equations (60) and (61) has a unique root in $U(n)$ by proving that
\[
\left(  2\pi n\right)  ^{2}+\frac{1}{2}\left[  \left(  Q_{n}+P_{n}^{\ast
}+A+A^{\prime}\right)  \pm\sqrt{\Delta\left(  \lambda\right)  }\right]
\]
is a contraction mapping. For this we show that there exist positive real
numbers $K_{1},K_{2},K_{3}$ such that%
\begin{equation}
\mid A\left(  \lambda\right)  -A(\mu)\mid<K_{1}\mid\lambda-\mu\mid,\text{
}\mid A^{\prime}(\lambda)-A^{\prime}(\mu)\mid<K_{2}\mid\lambda-\mu\mid,
\end{equation}%
\begin{equation}
\left\vert \sqrt{\Delta\left(  \lambda\right)  }-\sqrt{\Delta\left(
\mu\right)  }\right\vert <K_{3}\mid\lambda-\mu\mid,
\end{equation}
where $K_{1}+K_{2}+K_{3}<1$. The proof of (66) is similar to the proof of (56)
of the paper [26].

Now let us prove (67). By (62) and (8) we have
\[
\left(  \sqrt{\Delta\left(  \lambda\right)  }\right)  ^{-1}=o(1).
\]
On the other hand arguing as in the proof of (56) of the paper [26] we get
\[
\dfrac{d}{d\lambda}\Delta\left(  \lambda\right)  =O(1).
\]
Hence in any case we have
\[
\frac{d}{d\lambda}\sqrt{\Delta\left(  \lambda\right)  }=\frac{\dfrac
{d}{d\lambda}\Delta\left(  \lambda\right)  }{2\sqrt{\Delta\left(
\lambda\right)  }}=o(1).
\]
Thus by the fixed point theorem, each of the equations (60) and (61) has a
unique root $\lambda_{1}$ and $\lambda_{2}$ respectively. Clearly by (63) and
(64), we have $\lambda_{1}\neq\lambda_{2}$ which implies that the equation
(55) has two simple root in $U\left(  n\right)  .$ Therefore by Theorem 1(b),
$\lambda_{1}$ and $\lambda_{2}$ are the eigenvalues of $T_{1}$ lying in
$U\left(  n\right)  ,$ that is, they are $\lambda_{n,1}$ and $\lambda_{n,2}$,
which proves the simplicity of the large eigenvalues and the validity of (65).

If there exists a sequence $\left\{  n_{k}\right\}  $ such that (8) holds when
$n$ is replaced by $n_{k}$, then by Theorem 1(a)
\[
\left(  \Psi_{n_{k},1},\Psi_{n_{k},2}\right)  =1+O\left(  n_{k}^{-1/2}\right)
.
\]
Now it follows from the theorems of [20,21] (see also Lemma 3 of [24]) that
the root functions of $T_{1}$ do not form a Riesz basis.
\end{proof}

Now let us consider the operators $T_{2}$, $T_{3}$ and $T_{4}.$ First we
consider the operator $T_{3}$.

It is well-known that ( see formulas (47a), (47b)) in page 65 of [18] ) the
eigenvalues of the operators $T_{3}(q)$ consist of the sequences
$\{\lambda_{n,1,3}\},\{\lambda_{n,2,3}\}$ satisfying (10) when $\lambda_{n,j}$
is replaced by $\lambda_{n,j,3}.$ The eigenvalues, eigenfunctions and
associated functions of $T_{3}$ are
\begin{align*}
\lambda_{n} &  =\left(  2\pi n\right)  ^{2};\text{ }n=0,1,2,\ldots\\
y_{0}\left(  x\right)   &  =x-\dfrac{\alpha}{1+\alpha},\text{ }y_{n}\left(
x\right)  =\sin2\pi nx;\text{ }n=1,2,\ldots\\
\phi_{n}\left(  x\right)   &  =\left(  x-\dfrac{\alpha}{1+\alpha}\right)
\frac{\cos2\pi nx}{4\pi n};\text{ }n=1,2,\ldots.
\end{align*}
respectively. The biorthogonal systems analogous to (16), (17) are%
\begin{equation}
\left\{  \cos2\pi nx,\frac{4\left(  1+\overline{\alpha}\right)  }%
{1-\overline{\alpha}}\left(  \dfrac{1}{1+\overline{\alpha}}-x\right)  \sin2\pi
nx\right\}  _{n=0}^{\infty}%
\end{equation}%
\begin{equation}
\left\{  \sin2\pi nx,\frac{4\left(  1+\alpha\right)  }{1-\alpha}\left(
x-\dfrac{\alpha}{1+\alpha}\right)  \cos2\pi nx\right\}  _{n=0}^{\infty}%
\end{equation}
respectively.

Analogous formulas to (18) and (19) are%
\begin{equation}
\left(  \lambda_{N,j}-\left(  2\pi n\right)  ^{2}\right)  \left(  \Psi
_{N,j},\cos2\pi nx\right)  =\left(  q\left(  x\right)  \Psi_{N,j},\cos2\pi
nx\right)
\end{equation}%
\begin{equation}
\left(  \lambda_{N,j}-\left(  2\pi n\right)  ^{2}\right)  \left(  \Psi
_{N,j},\varphi_{n}^{\ast}\right)  -\gamma_{3}n\left(  \Psi_{N,j},\cos2\pi
nx\right)  =\left(  q\left(  x\right)  \Psi_{N,j},\varphi_{n}^{\ast}\right)
\end{equation}
respectively, where
\[
\gamma_{3}=\frac{16\pi\left(  1+\alpha\right)  }{1-\alpha}.
\]

Instead of (16)-(19) using (68)-(71) and arguing as in the proofs of Theorem 1
and Theorem 2 we obtain the following results for $T_{3}.$

\begin{theorem}
If (8) holds, then the large eigenvalues $\lambda_{n,j,3}$ are simple and
satisfy the following asymptotic formulas%
\begin{equation}
\lambda_{n,j,3}=\left(  2\pi n\right)  ^{2}+\left(  -1\right)  ^{j}\frac
{\sqrt{2\gamma_{3}}}{2}\sqrt{ns_{2n}}(1+o(1)).
\end{equation}
for $j=1,2.$ The eigenfunctions $\Psi_{n,j,3}$ corresponding to $\lambda
_{n,j,3}$ obey
\begin{equation}
\Psi_{n,j,3}=\sqrt{2}\sin2\pi nx+O\left(  n^{-1/2}\right)  .
\end{equation}
Moreover, if there exists a sequence $\left\{  n_{k}\right\}  $ such that (8)
holds when $n$ is replaced by $n_{k},$ then the root functions of $T_{3}$ do
not form a Riesz basis.
\end{theorem}

Now let us consider the operator $T_{2}$. It is well-known that ( see formulas
(47a), (47b)) in page 65 of [18] ) the eigenvalues of the operators $T_{2}(q)$
consist of the sequences $\{\lambda_{n,1,2}\},\{\lambda_{n,2,2}\}$ satisfying
\begin{equation}
\lambda_{n,j,2}=(2n\pi+\pi)^{2}+O(n^{1/2}),
\end{equation}
for $j=1,2$. The eigenvalues, eigenfunctions and associated functions of
$T_{2}$ are
\begin{align*}
\lambda_{n} &  =\left(  \pi+2\pi n\right)  ^{2},\text{ }y_{n}\left(  x\right)
=\cos\left(  2n+1\right)  \pi x,\\
\phi_{n}\left(  x\right)   &  =\left(  \frac{\beta}{\beta-1}-x\right)
\frac{\sin\left(  2n+1\right)  \pi x}{2\left(  2n+1\right)  \pi}%
\end{align*}
for $n=0,1,2,\ldots$respectively. The biorthogonal systems analogous to (16),
(17) are
\begin{equation}
\left\{  \sin\left(  2n+1\right)  \pi x,\frac{4\left(  \overline{\beta
}-1\right)  }{\overline{\beta}+1}\left(  x+\dfrac{1}{\overline{\beta}%
-1}\right)  \cos\left(  2n+1\right)  \pi x\right\}  _{n=0}^{\infty}%
\end{equation}%
\begin{equation}
\left\{  \cos\left(  2n+1\right)  \pi x,\frac{4\left(  \beta-1\right)  }%
{\beta+1}\left(  \frac{\beta}{\beta-1}-x\right)  \sin\left(  2n+1\right)  \pi
x\right\}  _{n=0}^{\infty}%
\end{equation}
respectively.

Analogous formulas to (18) and (19) are%
\begin{equation}
\left(  \lambda_{N,j}-\left(  \left(  2n+1\right)  \pi\right)  ^{2}\right)
\left(  \Psi_{N,j},\sin\left(  2n+1\right)  \pi x\right)  =\left(  q\left(
x\right)  \Psi_{N,j},\sin\left(  2n+1\right)  \pi x\right)
\end{equation}%
\begin{equation}
\left(  \lambda_{N,j}-\left(  \left(  2n+1\right)  \pi\right)  ^{2}\right)
\left(  \Psi_{N,j},\varphi_{n}^{\ast}\right)  -\left(  2n+1\right)  \gamma
_{2}\left(  \Psi_{N,j},\sin\left(  2n+1\right)  \pi x\right)  =\left(
q\left(  x\right)  \Psi_{N,j},\varphi_{n}^{\ast}\right)
\end{equation}
respectively, where
\[
\gamma_{2}=\frac{8\pi\left(  \beta-1\right)  }{\beta+1}.
\]
Instead of (16)-(19) using (75)-(78) and arguing as in the proofs of Theorem 1
and Theorem 2 we obtain the following results for $T_{2}.$

\begin{theorem}
If (9) holds, then the large eigenvalues $\lambda_{n,j,2}$ are simple and
satisfy the following asymptotic formulas%
\begin{equation}
\lambda_{n,j,2}=\left(  \left(  2n+1\right)  \pi\right)  ^{2}+\left(
-1\right)  ^{j}\frac{\sqrt{2\gamma_{2}}}{2}\sqrt{\left(  2n+1\right)
s_{2n+1}}(1+o(1)).
\end{equation}
for $j=1,2.$ The eigenfunctions $\Psi_{n,j,2}$ corresponding to $\lambda
_{n,j,2}$ obey
\begin{equation}
\Psi_{n,j,2}=\sqrt{2}\cos\left(  2n+1\right)  \pi x+O\left(  n^{-1/2}\right)
.
\end{equation}
Moreover, if there exists a sequence $\left\{  n_{k}\right\}  $ such that (9)
holds when $n$ is replaced by $n_{k},$ then the root functions of $T_{2}$ do
not form a Riesz basis.
\end{theorem}

Lastly we consider the operator $T_{4}$. It is well-known that ( see formulas
(47a), (47b)) in page 65 of [18] ) the eigenvalues of the operators $T_{4}(q)$
consist of the sequences $\{\lambda_{n,1,4}\},\{\lambda_{n,2,4}\}$ satisfying
(74) when $\lambda_{n,j,2}$ is replaced by $\lambda_{n,j,4}.$ The eigenvalues,
eigenfunctions and associated functions of $T_{4}$ are
\begin{align*}
\lambda_{n} &  =\left(  \pi+2\pi n\right)  ^{2},\text{ }y_{n}\left(  x\right)
=\sin\left(  2n+1\right)  \pi x,\\
\phi_{n}\left(  x\right)   &  =\left(  \frac{\alpha}{1-\alpha}+x\right)
\frac{\cos\left(  2n+1\right)  \pi x}{2\left(  2n+1\right)  \pi}%
\end{align*}
for $n=0,1,2,\ldots$respectively. The biorthogonal systems analogous to (16),
(17) are
\begin{equation}
\left\{  \cos\left(  2n+1\right)  \pi x,\frac{4\left(  1-\overline{\alpha
}\right)  }{1+\overline{\alpha}}\left(  \dfrac{1}{1-\overline{\alpha}%
}-x\right)  \sin\left(  2n+1\right)  \pi x\right\}  _{n=0}^{\infty}%
\end{equation}%
\begin{equation}
\left\{  \sin\left(  2n+1\right)  \pi x,\frac{4\left(  1-\alpha\right)
}{1+\alpha}\left(  \dfrac{\alpha}{1-\alpha}+x\right)  \cos\left(  2n+1\right)
\pi x\right\}  _{n=0}^{\infty}%
\end{equation}
respectively.

Analogous formulas to (18) and (19) are
\begin{equation}
\left(  \lambda_{N,j}-\left(  \pi+2\pi n\right)  ^{2}\right)  \left(
\Psi_{N,j},\cos\left(  2n+1\right)  \pi x\right)  =\left(  q\left(  x\right)
\Psi_{N,j},\cos\left(  2n+1\right)  \pi x\right)  ,
\end{equation}%
\begin{equation}
\left(  \lambda_{N,j}-\left(  \left(  2n+1\right)  \pi\right)  ^{2}\right)
\left(  \Psi_{N,j},\varphi_{n}^{\ast}\right)  -\left(  2n+1\right)  \gamma
_{4}\left(  \Psi_{N,j},\cos\left(  2n+1\right)  \pi x\right)  =\left(
q\left(  x\right)  \Psi_{N,j},\varphi_{n}^{\ast}\right)
\end{equation}
respectively, where
\[
\gamma_{4}=\frac{8\pi\left(  1-\alpha\right)  }{1+\alpha}.
\]
Instead of (16)-(19) using (81)-(84) and arguing as in the proofs of Theorem 1
and Theorem 2 we obtain the following results for $T_{4}.$

\begin{theorem}
If (9) holds, then the large eigenvalues $\lambda_{n,j,4}$ are simple and
satisfy the following asymptotic formulas%
\begin{equation}
\lambda_{n,j,4}=\left(  \left(  2n+1\right)  \pi\right)  ^{2}+\left(
-1\right)  ^{j}\frac{\sqrt{2\gamma_{4}}}{2}\sqrt{\left(  2n+1\right)
s_{2n+1}}(1+o(1)).
\end{equation}
for $j=1,2.$ The eigenfunctions $\Psi_{n,j,4}$ corresponding to $\lambda
_{n,j,4}$ obey
\begin{equation}
\Psi_{n,j,4}=\sqrt{2}\sin\left(  2n+1\right)  \pi x+O\left(  n^{-1/2}\right)
.
\end{equation}
Moreover, if there exists a sequence $\left\{  n_{k}\right\}  $ such that (9)
holds when $n$ is replaced by $n_{k},$ then the root functions of $T_{4}$ do
not form a Riesz basis.
\end{theorem}

\begin{remark}
Suppose that
\begin{equation}
\int_{0}^{1}xq\left(  x\right)  dx\neq0.
\end{equation}
If
\begin{equation}
\frac{1}{2}s_{2n}+B=o\left(  \frac{1}{n}\right)  ,
\end{equation}
where $B$ is defined by (34), then arguing as in the proof of Theorem 2, we
obtain that the large eigenvalues of the operator $T_{1}$ are simple. Moreover
if there exists a sequence $\left\{  n_{k}\right\}  $ such that (88) holds
when $n$ is replaced by $n_{k},$ then the root functions of $T_{1}$ do not
form a Riesz basis. The similar results can be obtained for the operators
$T_{2},T_{3}$ and $T_{4}.$
\end{remark}

\begin{remark}
Using (31) and (35) and arguing as in the proof of Theorem 3 of [1] it can be
obtained asymptotic formulas of arbitrary order for the eigenvalues and
eigenfunctions of the operator $T_{1}.$ The similar formulas can be obtained
for the operators $T_{2},T_{3}$ and $T_{4}.$
\end{remark}

\end{document}